\documentclass{amsart}
\usepackage[centertags]{amsmath}
\usepackage{amsfonts}
\usepackage{amssymb}
\usepackage{amsthm}
\usepackage{setspace}
\usepackage{graphicx}
\usepackage[shortlabels]{enumitem}
\usepackage{mathrsfs}

\theoremstyle{plain}
\newtheorem{thm}{Theorem}[section]

\newtheorem{lem}[thm]{Lemma}

\theoremstyle{definition}

\numberwithin{equation}{section}

\begin{document}
	\title[Third Hankel determinant for analytic functions]{The sharp bound of the third order Hankel determinant for  inverse of Ozaki close-to-convex functions}
	
	\author[B. Rath, K. S. Kumar, D. V. Krishna]{Biswajit Rath$^{1}$, K. Sanjay Kumar $^{2},$ D. Vamshee Krishna$^{3}$}
	
	\address{$^{1.2.3.4}$Department of Mathematics,
		Gitam Institute of Science, GITAM University,
		Visakhapatnam- 530 045, A.P., India}
	
	\email{brath@gitam.edu$^{1*}$,skarri9@gitam.in$^{2}$,vamsheekrishna1972@gmail.com$^{3}$}
	
	\begin{abstract}
		Let $f$ be analytic in the unit disk $\mathbb{D}= \{z \in \mathbb{C}~:~ |z| < 1\}$, and $\mathcal{S}$ be the subclass of normalized univalent functions given by $f(z)=\sum_{n=1}^{\infty}a_{n}z^{n},~a_{1}:=1$
		for $z \in\mathbb{D}$. We present the sharp bounds of the third order Hankel determinant for inverse functions when it belongs to of the class of Ozaki close-to-convex.
	\end{abstract}
	
	
	
	\date{}
	
	\maketitle
	\section{Introduction}
	Let $\mathcal{A}$ be the family of all analytic normalized mappings $f$ of the form
	\begin{equation}
		f(z)=\sum_{t=1}^{+\infty}a_{t}z^{t},~a_{1}:=1,
	\end{equation}
	in the open unit disc $\mathbb{D}=\{z \in \mathbb{C}:|z|<1\}$ and $\mathcal{S}$ is the subfamily of $\mathcal{A},$ possessing univalent (schlicht) mappings. For $f\in\mathcal{S},$ has an inverse $f^{-1}$ given by
	\begin{equation}
		f^{-1}(w)=w+\sum_{n=2}^{\infty}t_nw^n,~~|w|<r_o(f);\left(r_o(f)\geq\frac{1}{4}\right).
	\end{equation}
	A typical problem in geometric function theory is to study a functional made up of combination of the coefficients of the original functions. For the positive integers $r,~n,$ Pommerenke \cite{Pommerenke1966} characterized the $r^{th}$- Hankel determinant of $n^{th}$-order for $f$ given in (1.1), defined as follows:
	\begin{equation}
		H_{r,n}(f)=  \begin{array}{|cccc|}
			a_{n} & a_{n+1} & \cdots & a_{n+r-1} \\
			a_{n+1} & a_{n+2} & \cdots & a_{n+r} \\
			\vdots& \vdots & \ddots & \vdots \\
			a_{n+r-1} & a_{n+r} & \cdots & a_{n+2r-2}
		\end{array}.
	\end{equation}
	This determinant is investigated by many authors for different combinations of $r$ and $n$ in (1.3) which yields various types of Hankel determinants such as  $r = 2$ and $n =1$ gives the famous Fekete-Szeg$\ddot{o}$ functional mathematically denoted by  $|H_{2,1}(f)|:=|a_3-a_2^2|.$ The estimation of the sharp bounds to the functional $|H_{2,2}(f)|:=|a_{2}a_{4}-a_{3}^{2}|$ is obtained for  $r = n=2$ in (1.3) is called as second order Hankel determinant for bounded turning, starlike and convex, symbolized as $\Re,$  $S^{\ast}$ and $\mathcal{C}$ respectively fulfilling the conditions Re$\left\{f'(z)\right\}>0,$ Re$\left\{\frac{zf'(z)}{f(z)}\right\}>0$ and Re$\left\{1+\frac{zf''(z)}{f'(z)}\right\}>0$ in the unit disc $\mathbb{D},$  were obtained by many authors which can be found in the literature (see e.g., \cite{Janteng2007,Janteng2006}) .
	
	The problem of finding the sharp estimates of the third Hankel determinant obtained for $r = 3$ and $n = 1$ in (1.3) gives
	\begin{equation}
		H_{3,1}(f):= \begin{array}{|ccc|}
			a_{1}=1 & a_{2} & a_{3}\\
			a_{2} & a_{3} & a_{4}\\
			a_{3} & a_{4} & a_{5}
		\end{array}= 2a_2a_3a_4-a_3^3-a_4^2+a_3a_5-a_2^2a_5
	\end{equation}
	which is technically much tough than $r=n=2$.\\
	
	In recent years many authors are working on obtaining upper bounds ~~~(see \cite{Babalola2010,Kwon2019-1,Sim2021,Srivastava2021,Zaprawa2021}) and a few papers were devoted to the estimation of sharp upper bound to $H_{3,1}(f)$ for certain subclasses of analytic functions (see \cite{Banga2020,Kowalczyk2018-1,Kowalczyk2018-2,Kwon2019-2,Rath2022,Ullah2021}).
	
	A function $f\in\mathcal{A}$ belongs to $\mathcal{K}$, the class of close-to-convex functions if,
	and only if, there exists $g\in S^{\ast}$, such that Re$\{e^{i\tau}(zf'(z)/g(z))\}>0$ for $z\in\mathbb{D}$, and $\tau\in(-\pi/2,\pi/2)$. The class $\mathcal{K}$ was first formally introduced by Kaplan in
	1952 \cite{Kaplan}, who showed that $\mathcal{K}\subset \mathcal{S}$, so that $\mathcal{C}\subset \mathcal{S}^*\subset\mathcal{K}\subset \mathcal{S}$.
	In 1941 Ozaki \cite{ozaki1941} introduced a class of function known as Ozaki close to convex function is denoted by $\mathcal{F}$ and defined as
	\begin{equation}
		\mathcal{F}=\left\{f\in\mathcal{A}~\bigg|~\text{Re}\left(1+\frac{zf''(z)}{f'(z)}\right)>-\frac{1}{2}\right\}.
	\end{equation}
	
	We also note that $\mathcal{F}\subset\mathcal{K}$ follows from the original definition of Kaplan \cite{Kaplan}, and that Umezawa \cite{Umezawa} subsequently proved that functions in
	$\mathcal{F}$ are not necessarily starlike, but are convex in one direction.
	Recently, Bansal et al. \cite{bansal2015} obtained an upper bound to $H_{3,1}(f)$ for the class $\mathcal{F}$ as $\big(180+69\sqrt{15}\big)/32\sqrt{15}$ and  Maharana et al. \cite{Maharana2020} obtained an upper bound to $H_{3,1}(f^{-1})$ as $\big(45693+3648\sqrt{78}\big)/9216$ when $f\in \mathcal{F}$. Very recently,B.Kowalczyk et al\cite{Kowalczyk2023} obtainted sharp bound to $H_{3,1}(f)$ as $1/16$ when $f\in \mathcal{F}$.  
	
	Motivated by the results obtained by the authors mentioned above, in this paper we are making an attempt to estimate sharp bound for the third Hankel determinant namely $|H_{3,1}(f^{-1})|$, when $f$ belongs to the class of $\mathcal{F}$ and it is  invariance. 
	
	Let $\mathcal{P}$
	be a class of all functions $p$ having a positive real part in  $\mathbb{D}$:
	\begin{equation}
		p(z) = 1+\sum_{t=1}^{\infty}c_{t}z^{t}.
	\end{equation}
	Every such a function is called the Carath\'{e}odory function. In view of (1.4) and (1.5), the coefficients of the functions in $\mathcal{F}$ can be expressed in terms of coefficients of functions in $\mathcal{P}$. Hence, to estimate the upper bound of $|H_{3,1}(f^{-1})|,$ we build our computation on the well known formulas on coefficients $c_2$ (see, \cite[p. 166]{Pommerenke1975}, ), $c_3$ (see \cite{LibZlo1982,Libera1983}) and $c_4$ can be found in \cite{Rath2022A}.
	
	The foundation for proof of our main result are the following lemmas and we adopt some ideas
	from Libera and Zlotkiewicz \cite{Libera1983}.
	\begin{lem}
		If $p\in\mathcal{P},$ is of the form (1.5) with $c_1\geq0,$ such that $c_1\in[0,2]$ then
		\begin{equation*}
			\begin{split}
				&2c_2=c_1^2+\nu\mu,
				\\
				&4c_3=c_1^3+2c_1\nu\mu-c_1\nu\mu^2+2\nu\left(1-|\mu|^2\right)\rho,
				\\
				\text{and}\\
				&\begin{split}
					8c_4=~&c_1^4+3c_1^2\nu\mu+\left(4-3c_1^2\right)\nu\mu^2+c_1^2\nu\mu^3+4\nu\left(1-|\mu|^2\right)\left(1-|\rho|^2\right)\psi\\&+4\nu\left(1-|\mu|^2\right)\left(c_1\rho-c\mu\rho-\bar{\mu}\rho^2\right),
				\end{split}
			\end{split}
		\end{equation*}
		where $\nu:=4-c_1^2,$  for some $\mu$, $\rho$ and $\psi$ such that $|\mu|\leq 1$, $|\rho|\leq 1$ and $|\psi|\leq1$.
	\end{lem}
\begin{lem}
	Let $\psi_1,\psi_2,\psi_3,\psi_4:[0,2]\rightarrow \mathbb{R}$ be a function definded by
	\begin{equation*}
		\begin{split}
			\psi_1(c):&=- 160 c^2 + 16 c^3 + 20 c^4 - 4 c^5+\frac{5 c^6}{4}\\\psi_2(c):&=32 c + 48 c^2 + 32 c^3 + 14 c^4 - 10 c^5-\frac{13 c^6}{2}\\\psi_3(c):&=-256 + 64 c + 276 c^2 - 48 c^3 - 82 c^4 + 8 c^5+\frac{29 c^6}{4} \\\psi_4(c):&=320 - 32 c - 272 c^2 - 32 c^3 + 76 c^4 + 10 c^5 - 7 c^6\\\psi_5(c):&=-64 - 64 c + 48 c^2 + 32 c^3 - 12 c^4 - 4 c^5 + c^6
		\end{split}
	\end{equation*}
	Then following are true
	
	\textbf{(a)} $\psi_1(c)\leq 0.$ for $c\in[0,2]$
	
	\textbf{(b)} $\psi_1(c)+\psi_2(c)\leq 0.$ for $c\in (\frac{87137}{250000},2]$
	
	\textbf{(c)} $\psi_1(c)+\psi_2(c)+\psi_3(c)\leq 0.$ for $c\in \left(\frac{87137}{250000},\frac{4511}{4000}\right]$
	
	\textbf{(d)} $\psi_1(c)+\psi_2(c)+\psi_3(c)+0.6\psi_4(c)\leq 0.$ for $c\in \left[\frac{4511}{4000},2\right]$
	
	\textbf{(e)} $\psi_5(c)\leq 0.$ for $c\in[0,2]$
		\begin{proof}[Proof(a)]
			Since $4 - c^2>0$ for $c\in[0,2]$
			\begin{equation*}
					\begin{split}
						\psi_1(c):&=- 160 c^2 + 16 c^3 + 20 c^4 - 4 c^5+\frac{5 c^6}{4}\\&=-48 c^2 - \frac{3 c^6}{4} - 2 c^2 (4 - c^2) (14 - 2 c + c^2)\leq 0
					\end{split}
			\end{equation*}
		\end{proof}
	\begin{proof}[Proof(b)]
		Let $c=\frac{87137 }{250000}t,1< t\leq\frac{500000}{87137}$. Then
		\begin{equation*}
			\begin{split}
				\psi_1(c(t))+\psi_2(c(t)):=&11.1535 t-13.6064 t^2+2.03249 t^3+0.501798 t^4\\&-0.072018 t^5-0.00941315 t^6\\\leq&0.0094t (1-t)(17.189 -7.9049 t+t^2) (69.0293 +16.5646 t+t^2)\leq 0
			\end{split}
		\end{equation*}
	\end{proof}
	\begin{proof}[Proof(c)]
	Let $c=\frac{87137 }{250000}t,1< t\leq\frac{563875}{174274}$. Then
	\begin{equation*}
		\begin{split}
			\psi_1(c(t))+\psi_2(c(t))+\psi_3(c(t)):=&-256.+33.4606 t+19.9237 t^2-0.708421 t^4\\&-0.0308649 t^5+0.00358596 t^6\\\leq&0.00358596(-18.403+t)(8.71052 +t)\\&(15.6059 -7.89366 t+t^2) (28.5372 +8.97902 t+t^2)\leq 0
		\end{split}
	\end{equation*}
\end{proof}
	\begin{proof}[Proof(d)]
	Let $c=\frac{4511}{4000}t,1\leq t\leq\frac{8000}{4511}$. Then
	\begin{equation*}
		\begin{split}
			\psi_1+\psi_2+\psi_3+06\psi_4=&64+72.176 t-137.357 t^2-45.8974 t^3\\&+45.2907 t^4+7.29666 t^5-10.286 t^6\\\leq&10.286 (1-t) (0.511168 +t)\\&(4.07379-3.46659 t+t^2)(3.06843 +3.21981 t+t^2)\leq 0
		\end{split}
	\end{equation*}
\end{proof}
\begin{proof}[Proof(e)]
	Since $4 - c^2>0$ for $c\in[0,2]$
	\begin{equation*}
		\begin{split}
			\psi_1=(4 - c^2)^2(-4-4 c+c^2)\leq 0
		\end{split}
	\end{equation*}
\end{proof}
\end{lem}
\begin{lem}
	Let $\Psi:[0,\frac{87137}{250000}]\times(0,0.25)\rightarrow \mathbb{R}$ be a function definded by
	\begin{equation*}
			\Psi(c,x)=320+\psi_1(c)+\psi_2(c)x+\psi_3(c)x^2+\psi_4(c)x^3+\psi_5(c)x^4
	\end{equation*}
where $\psi_1,\psi_2,\psi_3,\psi_4,,\psi_5$ define as lemma 1.2 for $c\in[0,\frac{87137}{250000}]$,
Then $\Psi(c,x)\leq 320 $ for $0\leq c\leq\frac{87137}{250000}$ and $0<x<0.25$
\begin{proof}
	Since $\phi_4(c)>0$ and  $\phi_5(c)<0$ in $0\leq c\leq\frac{87137}{250000}$
	\begin{equation*}
		\begin{split}
			\Psi(c,x)\leq&320+\psi_1(c)+\psi_2(c)x+\left(\psi_3(c)+0.25\psi_4(c)\right)x^2\\:=&h(c,x),~\text{with}~c\in\left[0,\frac{87137}{250000}\right]~\text{and}~x\in(0,0.25).
		\end{split}
	\end{equation*}
$\partial h/\partial x=0$ iff

\begin{align*}
x= \frac{-64 c - 96 c^2 - 64 c^3 - 28 c^4 + 20 c^5 + 13 c^6}{-704 + 224 c + 832 c^2 - 224 c^3 - 252 c^4 + 42 c^5 + 22 c^6}:=x_0\in(0,0.25)
\end{align*}

and

\begin{align*}
\frac{\partial^2\Psi}{\partial x^2}(c,x_0)=-\left(4-c^2\right)\left(88 - 28 c - 82 c^2 + 21 c^3 + 11 c^4\right)<0
\end{align*}

Therefore  $\Psi(c,x)$ attains maximum at $(c,x_0).$

Hence
\begin{equation*}
\Psi(c,x)\leq\Psi(c,x_0)=
\frac{\begin{split}&225280 - 71680 c - 321536 c^2 + 103936 c^3 \\&+ 148224 c^4 - 39936 c^5 - 
		19856 c^6 + 7816 c^7 \\&- 80 c^8 - 662 c^9 - 59 c^{10}\end{split}}{8\left(88 - 28 c - 82 c^2 + 21 c^3 + 11 c^4\right)^2}
<320
	\end{equation*}
\end{proof}
\end{lem}
\begin{lem}
	Let $\Psi:\left[0,\frac{87137}{250000}\right]\times[0.25,1]\rightarrow \mathbb{R}$ be a function definded by
	\begin{equation*}
		\Phi(c,x)=\phi_1(x)+\phi_2(x)c+\phi_3(x)c^2+\phi_4(x)c^3+\phi_5(x)c^4+\phi_6(x)c^5+\phi_7(x)c^6
	\end{equation*}
	where for $x\in[0.25,1]$,
	\begin{equation*}
		\begin{split}
			\phi_1(x):&=-256 x^2 + 320 x^3 - 64 x^4\\\phi_2(x):&=32 x + 64 x^2 - 32 x^3 - 64 x^4\\\phi_3(x):&=-160 + 48 x + 276 x^2 - 272 x^3 + 48 x^4 \\\phi_4(x):&=16 + 32 x - 48 x^2 - 32 x^3 + 32 x^4\\\phi_5(x):&=20 + 14 x - 82 x^2 + 76 x^3 - 12 x^4\\\phi_6(x):&=-4 - 10 x + 8 x^2 + 10 x^3 - 4 x^4\\\phi_7(x):&=\frac{5}{4} - \frac{13 x}{2} +\frac{29 x^2}{4} - 7 x^3 + x^4
		\end{split}
	\end{equation*}
	Then $\Phi(c,x)< 0 $ for $0\leq c\leq\frac{87137}{250000}$ and $0.25\leq x<1$
	\begin{proof}
		Since $\phi_1,\phi_1+\phi_2,\phi_1+\phi_2+\phi_3,\phi_1+\phi_2+\phi_3+\phi_4,\phi_1+\phi_2+\phi_3+\phi_4+\phi_5,\phi_6,\phi_7\leq0$ for $0.25\leq x<1$
		\begin{equation*}
			\Phi(c,x)<\left(\phi_1(x)+\phi_2(x)+\phi_3(x)+\phi_4(x)+\phi_5(x)\right)c^4+\phi_6(x)c^5+\phi_7(x)c^6<0
		\end{equation*}
	\end{proof}
\end{lem}
\begin{lem}
	Let $\Psi:\left(\frac{87137}{250000},\frac{4511}{4000}\right]\times(0,0.6)\rightarrow \mathbb{R}$ be a function definded as in lemma 1.2.
	Then $\Psi(c,x)\leq 320 $ for $\frac{87137}{250000}< c\leq \frac{4511}{4000}$ and $0<x<0.6$
	\begin{proof}
		Since $\phi_4(c)>0,\phi_1(c),\phi_1(c)+\phi_2(c)<0,\phi_1(c)+\phi_2(c)+\phi_3(c)<0,\phi_1(c)+\phi_2(c)+\phi_3(c)+0.6\phi_4(c)<0$ and  $\phi_5(c)<0$ in for $\frac{87137}{250000}< c\leq \frac{4511}{4000}$
		\begin{equation*}
			\begin{split}
				\Psi(c,x)\leq&320+(\phi_1(c)+\phi_2(c)+\phi_3(c)+0.6\phi_4(c))x^2+\phi_5(c)x^4<320.
			\end{split}
		\end{equation*}
	\end{proof}
\end{lem}
\begin{lem}
	Let $\Gamma:\left(\frac{87137}{250000},1\right]\times[0.6,1]\rightarrow \mathbb{R}$ be a function definded by
	\begin{equation*}
		\Gamma(c,x)=\gamma_1(x)+\gamma_2(x)c+\gamma_3(x)c^2+\gamma_4(x)c^3+\gamma_5(x)c^4+\gamma_6(x)c^5+\phi_7(x)c^6
	\end{equation*}
	where for $x\in[0.6,1]$,
	\begin{equation*}
		\begin{split}
			\gamma_1(x):&=-256 x^2 + 320 x^3 - 64 x^4\\\gamma_2(x):&= 96 x^2 - 32 x^3 - 64 x^4\\\gamma_3(x):&=164 x^2 - 272 x^3 + 48 x^4 \\\gamma_4(x):&=- 32 x^3 + 32 x^4\\\gamma_5(x):&=- 48 x^2 + 76 x^3 - 12 x^4\\\gamma_6(x):&=-6x^2 + 10 x^3 - 4 x^4\\\gamma_7(x):&=2 x^2 - 7 x^3 + x^4
		\end{split}
	\end{equation*}
	Then $\Phi(c,x)< 0 $ for $\frac{87137}{250000}< c\leq1$ and $0.6\leq x\leq1$
	\begin{proof}
		Since $\gamma_1,\gamma_1+\gamma_2,\gamma_1+\gamma_2+\gamma_3,\gamma_4,\gamma_1+\gamma_2+\gamma_3+\gamma_4+\gamma_5,\gamma_6,\gamma_7\leq0$ for $0.6\leq x<1$
		\begin{equation*}
			\Phi(c,x)<\left(\gamma_1+\gamma_2+\gamma_3+\gamma_4+\gamma_5\right)c^4+\gamma_6(x)c^5+\gamma_7(x)c^6<0
		\end{equation*}
	\end{proof}
\end{lem}
\begin{lem}
	Let $\Psi:\left(1,\frac{4511}{4000}\right]\times[0.6,1]\rightarrow \mathbb{R}$ be a function definded as in lemma 1.2.
	Then $\Psi(c,x)\leq 320 $ for $1< c\leq \frac{4511}{4000}$ and $0.6\leq x\leq1$
	\begin{proof}
		Since $\phi_1(c),\phi_1(c)+\phi_2(c)<0,\phi_1(c)+\phi_2(c)+\phi_3(c)<0$  for $1< c\leq 1.12775$
		\begin{equation*}
			\begin{split}
				\Psi(c,x)\leq&320+(\phi_1(c)+\phi_2(c)+\phi_3(c))x^2+\phi_4(c)x^3+\phi_5(c)x^4\\=&320+\left(-256 + 96 c + 164 c^2 - 48 c^4 - 6 c^5 + 2 c^6\right) x^2\\&+\left(320 - 32 c - 272 c^2 - 32 c^3 + 76 c^4 + 10 c^5 - 7 c^6\right) x^3\\&+\left(-64 - 64 c + 48 c^2 + 32 c^3 - 12 c^4 - 4 c^5 + c^6\right) x^4\\\leq&320 - 23 x^2 + 63 x^3 - 53 x^4\leq318.459.
			\end{split}
		\end{equation*}
	\end{proof}
\end{lem}
\begin{lem}
	Let $\Psi:(\frac{4511}{4000},2]\times[0,1]\rightarrow \mathbb{R}$ be a function definded as in lemma 1.2.
Then $\Psi(c,x)\leq 320 $ for $\frac{4511}{4000}< c\leq2$ and $0\leq x\leq1$
\begin{proof}
	Since $\phi_1(c),\phi_1(c)+\phi_2(c)<0,\phi_1(c)+\phi_2(c)+\phi_3(c)<0,\phi_1(c)+\phi_2(c)+\phi_3(c)+\phi_4(c),\phi_5(c)<0$  for $1.12775< c<2$
	\begin{equation*}
		\begin{split}
			\Psi(c,x)\leq&320+(\phi_1(c)+\phi_2(c)+\phi_3(c)+\phi_4(c))x^4+\phi_5(c)x^4<320.
		\end{split}
	\end{equation*}
\end{proof}
\end{lem}
	\section{\textbf{Bound for $H_{3,1}(f^{-1})$}}
	\begin{thm}
		If $f\in$ $\mathcal{F}$, then
		\begin{equation*}
			\big|H_{3,1}(f^{-1})\big| \leq \frac{1}{16}
		\end{equation*} 
		and the inequality is sharp for $f_0=z/\sqrt{1-z^2}$.
	\end{thm}
	\begin{proof}
		For $f\in$ $\mathcal{F}$, using the defination (1.2), we have 
		\begin{equation}
			w=f(f^{-1})=f^{-1}(w)+\sum_{n=2}^{\infty}a_n(f^{-1}(w))^n.
		\end{equation}
		Further, we have
		\begin{equation}
			w=f(f^{-1})=w+\sum_{n=2}^{\infty}t_nw^n+\sum_{n=2}^{\infty}a_n(w+\sum_{n=2}^{\infty}t_nw^n)^n.
		\end{equation}
		Upon simplification, we obtain
		\begin{align}\nonumber
			(t_2+a_2)w^2+(t_3+2a_2t_2+a_3)w^3+(t_4+2a_2t_3+a_2t_2^2+3a_3t_2+a_4)w^4\\+(t_5+2a_2t_4+2a_2t_2t_3+3a_3t_3+3a_3t_2^2+4a_4t_2+a_5)w^5+......=0.
		\end{align}
		Equating the coefficients in power of $w$ from (2.3), upon simplification, we obtain
		\begin{equation}
			\begin{split}
				&t_2=-a_2; t_3=\{-a_3+2a_2^2\}; t_4=\{-a_4+5a_2a_3-5a_2^3\};\\&t_5=\{-a_5+6a_2a_4-21a_2^2a_3+3a_3^2+14a_2^4\}.
			\end{split}
		\end{equation}
		Using the values of $a_n (n=2,3,4,5)$ from (2.2) in (2.4), upon simplification, we obtain
		\begin{equation}
			\begin{split}
				&t_2=-\frac{3c_1}{4},~t_3=\frac{1}{4}\left(3c_1^2-c_2\right),~t_4=-\frac{1}{32}\left(27c_1^3-21c_1c_2+4c_3\right)\\\text{and}~
				&t_5=-\frac{3}{160}\left(4c_4-22c_1c_3+69c_1^2c_2-7c_2^2-54c_1^4\right).
			\end{split}
		\end{equation}
		Now,
		\begin{align}
			H_{3,1}(f^{-1})&= \begin{array}{|ccc|}
				t_{1}=1 & t_{2} & t_{3}\\
				t_{2} & t_{3} & t_{4}\\
				t_{3} & t_{4} & t_{5}
			\end{array}~,
		\end{align}
		Using the values of $t_j,~ (j=2,3,4,5)$ from (2.5) in (2.6), it simplifies to give
		\begin{equation}
			\begin{split}
				H_{3,1}(f^{-1})=\frac{1}{5120}\Big(& 27c_1^6-108c_1^4c_2+36c_1^3c_3+117c_1^2c_2^2-88c_2^3\\&+72c_1c_2c_3-72c_1^2c_4-80c_3^2+96c_2c_4\Big).
			\end{split}
		\end{equation}
		In view of Lemma 1.1, adopting the similar procedure as Theorem 2.2 in (2.7), we obtain
		\begin{equation}
			\begin{split}
				H_{3,1}(f^{-1})=\frac{1}{5120}&\bigg[\frac{5 c_1^6}{4} + \nu \bigg\{-\frac{13 c_1^4 \mu}{2} - 
				2 c_1^4 \mu^2 + \left(37 c_1^2 - \frac{37 c_1^4}{4}\right) \mu^2+  c_1^2\nu \mu^4 \\&- \left(\frac{236}{7} + 
				7 \left(-\frac{18}{7} + c_1^2\right)^2\right) \mu^3  + (1 - 
				\mu^2) \Big(2 c_1 \nu (1 - 2 \mu) \mu + 4 c_1^3 (1 + 3 \mu)\Big) \rho\\&+ 
				4 (1 - \mu^2) \Big(3 c_1^2 \mu - \nu (5 + \mu^2)\Big) \rho^2 + 
				12 \Big(-c_1^2 + 2 \nu \mu\Big) (1 - \mu^2) (1 - \rho^2)\bigg\}\bigg].
			\end{split}
		\end{equation}
		Putting $c:=c_1$ and using $\nu=(4-c^2)$ in (2.8), a simple calculation gives
		\begin{equation}
			\begin{split}
				H_{3,1}(f^{-1})=\frac{1}{5120}&\bigg[\frac{5 c^6}{4} + (4 - c^2) \bigg\{-\frac{13 c^4 \mu}{2} - 
				2 c^4 \mu^2 + \left(37 c^2 - \frac{37 c^4}{4}\right) \mu^2+ (4 c^2 - c^4) \mu^4 \\&- \left(\frac{236}{7} + 
				7 \left(-\frac{18}{7} + c^2\right)^2\right) \mu^3  + (1 - 
				\mu^2) \Big(2 c (4 - c^2) (1 - 2 \mu) \mu + 4 c^3 (1 + 3 \mu)\Big) \rho\\&+ 
				4 (1 - \mu^2) \Big(3 c^2 \mu - (4 - c^2) (5 + \mu^2)\Big) \rho^2 + \\&
				12 \Big(-c^2 + 2 (4 - c^2) \mu\Big) (1 - \mu^2) (1 - \rho^2)\bigg\}\bigg].
			\end{split}
		\end{equation}
		Taking modulus on both sides of (2.9), using $|\mu|=x\in[0,1]$, $|\rho|=y\in[0,1]$, $c_1=c\in[0,2]$ and $|\psi|\leq 1$, we obtain
		\begin{equation}
			\bigg|H_{3,1}(f^{-1})\bigg|\leq\frac{1}{5120}\vartheta  \left(c,x,y\right),
		\end{equation}
		where $\vartheta:\mathbb{R}^3\rightarrow\mathbb{R}$ is defined as
		\begin{equation}
			\begin{split}
				\vartheta  \left(c,x,y\right)=&\bigg[\frac{5 c^6}{4} + (4 - c^2) \bigg\{\frac{13 c^4 x}{2} +
				2 c^4 x^2 + \left(37 c^2 - \frac{37 c^4}{4}\right) x^2+ (4 c^2 - c^4) x^4 \\&+ \left(\frac{236}{7} + 
				7 \left(-\frac{18}{7} + c^2\right)^2\right) x^3  + (1 - 
				x^2) \Big(2 c (4 - c^2) (1 + 2 x) x + 4 c^3 (1 + 3 x)\Big) w\\&+ 
				4 (1 - x^2) \Big(3 c^2 x + (4 - c^2) (5 + x^2)\Big) y^2 \\&+ 
				12 \Big(c^2 + 2 (4 - c^2) x\Big) (1 - x^2) (1 - y^2)\bigg\}\bigg].
			\end{split}
		\end{equation}
		To achive our result, it is sufficient to maximize the function $\vartheta\left(c,x,y\right)$ on \\$\Omega:=[0,2]\times[0,1]\times[0,1]$.\\
		{\bf A.} On the eight vertices of $\Omega$, from (2.11), we have
		\begin{equation*}
			\begin{split}
				&\vartheta  \left(0,0,0\right)=\vartheta  \left(2,0,0\right)=\vartheta  \left(2,1,0\right)=\vartheta  \left(2,1,1\right)=\vartheta  \left(2,0,0\right)=0,\\&\vartheta  \left(0,0,1\right)=320,~\vartheta  \left(0,1,0\right)=\vartheta  \left(0,1,1\right)=320.
			\end{split}
		\end{equation*}
		{\bf B.} Now, we consider each of the tweleve edges of $\Omega$ in view of (2.11) 
		
		as fallows 
		
		(i) For the edge $c=0,~x=0,~0<y<1,$ we obtain.
		\begin{equation*}
			\vartheta \left(0,0,y\right)=320 y^2\leq320.
		\end{equation*}
		%
		
		(ii) For the edge $c=0,~x=1,~0<y<1$, we obtain
		\begin{equation*}
			\vartheta \left(0,1,y\right)=320.
		\end{equation*}
		
		(iii) For $c=0,~y=0,~0<x<1$,
		\begin{equation*}
			\vartheta \left(0,x,0\right)=384 x- 64 x^3\leq320.
		\end{equation*}
		%
		
		(iv) For $c=0,~y=1,~0<x<1$,
		\begin{equation*}
			\begin{split}
				\vartheta \left(0,x,1\right)&=320 - 256 x^2 + 320 x^3 - 64 x^4\\&=320 - 64 (4 -x) (1 -x) x^2\leq320.
			\end{split}
		\end{equation*}
		
		%
		%
		
		(v) If $x=0,~y=0,~0<c<2$ in (2.11), then
		\begin{equation*}
			\vartheta \left(c,0,0\right)=48 c^2 - 12 c^4 + \frac{5 c^6}{4}\leq80.
		\end{equation*}
		
		
		(vi) For $x=0,~y=1,0<c<2$,
		\begin{equation*}
			\vartheta \left(c,0,1\right)=320 - 160 c^2 + 16 c^3 + 20 c^4 - 4 c^5 + \frac{5 c^6}{4}\leq320.
		\end{equation*}
		
		(vii) For the edges: $x=1,~y=0,0<c<2~\text{or}~x=1,~y=1,0<c<2$, we have
		\begin{equation*}
			\vartheta \left(c,1,y\right)=320 - 60 c^2 + 16 c^4 - 4 c^6\leq320.
		\end{equation*}
		%
		
		(viii) For the edges:  $c=2,~x=0,~0<y<1$ or $c=2,~x=1,~0<y<1$ or
		
		$c=2,~y=0,~0<x<1$ or $c=2,~y=1,~0<x<1$, we obtain
		\begin{equation*}
			\vartheta \left(2,x,y\right)=0.
		\end{equation*}
		{\bf C.} Further, we us consider the six faces of $\Omega$ with respect to $\vartheta  \left(c,x,y\right)$, from (2.11)
		
		(i) If $c=2$, in (2.11), then $$\vartheta  \left(2,x,y\right)=0.$$
		
		(ii) On the face $c=0$, from (2.11), we obtain
		\begin{equation*}
			\begin{split}
				\vartheta  \left(0,x,y\right)&=384 x- 64 x^3 + (320 - 384 x- 256 x^2 + 384 x^3 - 64 x^4) y^2\\&=384 x- 64 x^3 + 64 (5 -x) (1 -x)^2 (1 +x) y^2\\&\leq384 x- 64 x^3 + 64 (5 -x) (1 -x)^2 (1 +x)\\&=320 - 256 x^2 + 320 x^3 - 64 x^4\leq320
			\end{split}
		\end{equation*}
		%
		%
		
		(iii) On the face $x=0,~c\in(0,2),~y\in(0,1)$, from (2.11), we obtain
		\begin{equation*}
			\begin{split}
				\vartheta\left(c,0,y\right)&=\frac{5 c^6}{4} + (4 - c^2) (4 c^3 y+ 20 (4 - c^2) y^2 + 12 c^2 (1 - y^2))\\&=\frac{5 c^6}{4} + (4 - c^2) (4 c^3 y+ 80 y^2 + c^2 (-20 y^2 + 12 (1 - y^2)))\\&\leq\frac{5 c^6}{4} + (4 - c^2) (4 c^3  + 80  + 12c^2)\leq320,~c\in(0,2).
			\end{split}
		\end{equation*}
		%
		
		(vi) On the face $x=1,~c\in(0,2),~y\in(0,1)$ in (2.11), we observe that the 
		
		function $\vartheta  \left(c,1,y\right)$ is independent of $y$, from B(vii), we have
		\begin{equation*}
			\vartheta  \left(c,1,y\right)\leq320.
		\end{equation*}
		
		(v) For $y=0,~c\in(0,2),~x\in(0,1)$ in (2.11), we get
		
		\begin{equation*}
			\begin{split}
				\vartheta\left(c,x,0\right)=\frac{5 c^6}{4} + &(4 - c^2) (\frac{13 c^4 x}{2} + 
				2 c^4 x^2 + (37 c^2 - \frac{37 c^4}{4}) x^2 + (4 c^2 - c^4) x^4\\& + (\frac{236}{7} + 
				7 (-(\frac{18}{7}) + c^2)^2) x^3 + 
				12 (c^2 + 2 (4 - c^2) x) (1 - x^2))\\=\frac{5 c^6}{4} + &(4 - c^2) (96 x- 16 x^3 + c^4 (\frac{13 x}{2} - \frac{29 x^2}{4} + 7 x^3 - x^4) \\&+ 
				c^2 (12 - 24 x + 25 x^2 - 12 x^3 + 4 x^4))\\\leq\frac{5 c^6}{4} + &(4 - c^2)(80 + \frac{21}{4}c^4  + 12 c^2 )
				\leq320,~c\in(0,2),
			\end{split}
		\end{equation*}
			
			(vi) For $y=1$,
			\begin{align*}
				\vartheta  \left(c,x,1\right)=&320 - 160 c^2 + 16 c^3 + 20 c^4 - 4 c^5+\frac{5 c^6}{4}\\&+\left(32 c + 48 c^2 + 32 c^3 + 14 c^4 - 10 c^5-\frac{13 c^6}{2}\right)x\\&+\left(-256 + 64 c + 276 c^2 - 48 c^3 - 82 c^4 + 8 c^5+\frac{29 c^6}{4}\right) x^2\\&+\left(320 - 32 c - 272 c^2 - 32 c^3 + 76 c^4 + 10 c^5 - 7 c^6\right) x^3\\&+\left(-64 - 64 c + 48 c^2 + 32 c^3 - 12 c^4 - 4 c^5 + c^6\right) x^4\\&\text{Using lemma 1.2,1.3, 1.4, 1.5, 1.6, 1.7, 1.8 we can see that} \\\leq&320,~\text{with}~c\in(0,2)~\text{and}~x\in(0,1).
			\end{align*}

			{\bf D.} Now we consider the interior portion of $\Omega$, i.e. $(0,2)\times(0,1)\times(0,1).$
			
			Differentiating $\vartheta (c,x,y)$ given in (2.11) partially with respect $y$, we obtain
			\begin{align*}
				\frac{\partial \vartheta }{\partial y}=(4 - c^2)& \Big((1 - x^2) (2 c (4 - c^2) x (1 + 2 x) + 4 c^3 (1 + 3 x)) \\&- 
				24 (c^2 + 2 (4 - c^2) x) (1 - x^2) y+ 
				8 (1 - x^2) (3 c^2 x + (4 - c^2) (5 + x^2)) y\Big).
			\end{align*}
			
			Upon solving $\frac{\partial \vartheta }{\partial y}=0$, we get
			\begin{equation*}
				y_1=\frac{4 c x (1 + 2 x) + c^3 (2 + (5 - 2 x) x)}{4 (c^2 (-8 +x) - 4 (-5 +x)) (-1 +x)},
			\end{equation*}
			
			for $y_1>0,$ iff the following condition
			
			\begin{equation}
				c^2 (-8 +x) - 4 (-5 +x)<0
			\end{equation}
			
			must hold. From (2.12), we have
			\begin{equation}
				\left(\frac{4}{\sqrt{7}}<c\leq \sqrt{\frac{5}{2}}\land \frac{8 c^2-20}{c^2-4}<x<1\right)\lor \left(\sqrt{\frac{5}{2}}<c<2\land 0<x<1\right).
			\end{equation}
			Hence, $y_1<0,$ for $\left(0<c\leq \frac{4}{\sqrt{7}}\land 0<x<1\right)\lor \left(\frac{4}{\sqrt{7}}<c<\sqrt{\frac{5}{2}}\land 0<x<\frac{8 c^2-20}{c^2-4}\right).$
			Therefore,$\vartheta (c,x,y)$ has no critical point in the interior of $(0,\frac{4}{\sqrt{7}}]\times(0,1)\times(0,1)$.\\
			Now, it remains to maximize $\vartheta (c,x,y)$ in $(\frac{4}{\sqrt{7}},2)\times(0,1)\times(0,1).$ 
			
			We can rewrite $\vartheta (c,x,y)$ as
			\begin{equation*}
				\begin{split}
					\vartheta (c,x,y)=&\frac{5 c^6}{4}+\left(4-c^2\right) \Bigg[2 c^4 x^2+\frac{13 c^4 x}{2}+\left(7 \left(c^2-\frac{18}{7}\right)^2+\frac{236}{7}\right) x^3\\&+\left(4 c^2-c^4\right) x^4+\left(37 c^2-\frac{37 c^4}{4}\right) x^2\\&+\left(1-x^2\right)\left(4 c^3 (3 x+1)+2 \left(4-c^2\right) cx (2 x+1)\right)y\\&+ \left(1-x^2\right)  \left(4\left(4-c^2\right) \left(x^2+5\right)+12 c^2 x-12\left(2 \left(4-c^2\right) x+c^2\right)\right)y^2\\&+12 \left(1-x^2\right) \left(2 \left(4-c^2\right) x+c^2\right)\Bigg].
				\end{split}
			\end{equation*}
			\\ 
			\textbf{ Case D1} Suppose, $\left(4\left(4-c^2\right) \left(x^2+5\right)+12 c^2 x-12\left(2 \left(4-c^2\right) x+c^2\right)\right)>0.$ 
			
			Then, $	\vartheta (c,x,y)\leq\vartheta (c,x,1)<320.$
			\\
			\textbf{ Case D2} Suppose, $\left(4\left(4-c^2\right) \left(x^2+5\right)+12 c^2 x-12\left(2 \left(4-c^2\right) x+c^2\right)\right)\leq0.$	
			
			Then,
			\begin{equation*}
				\begin{split}
					\vartheta (c,x,y)\leq&\frac{5 c^6}{4}+\left(4-c^2\right) \Bigg[2 c^4 x^2+\frac{13 c^4 x}{2}+\left(7 \left(c^2-\frac{18}{7}\right)^2+\frac{236}{7}\right) x^3\\&+\left(4 c^2-c^4\right) x^4+\left(37 c^2-\frac{37 c^4}{4}\right) x^2\\&+\left(1-x^2\right)\left(4 c^3 (3 x+1)+2 \left(4-c^2\right) c x(2 x+1)\right)y\\&+12 \left(1-x^2\right) \left(2 \left(4-c^2\right) x+c^2\right)\Bigg]
				\end{split}
			\end{equation*}
			\begin{equation*}
				\begin{split}
					\leq&\frac{5 c^6}{4}+\left(4-c^2\right) \Bigg[2 c^4 x^2+\frac{13 c^4 x}{2}+\left(7 \left(c^2-\frac{18}{7}\right)^2+\frac{236}{7}\right) x^3\\&+\left(4 c^2-c^4\right) x^4+\left(37 c^2-\frac{37 c^4}{4}\right) x^2\\&+\left(1-x^2\right)\left(4 c^3 (3 x+1)+2 \left(4-c^2\right) c x (2 x+1)\right)\\&+12 \left(1-x^2\right) \left(2 \left(4-c^2\right) x+c^2\right)\Bigg]
					\\=&\frac{5 c^6}{4}-4 c^5-12 c^4+16 c^3+48 c^2\\&+\left(c^6-4 c^5-8 c^4+32 c^3+16 c^2-64 c\right) x^4\\&+\left(-7 c^6+10 c^5+40 c^4-32 c^3-32 c^2-32 c-64\right) x^3\\&+\left(\frac{29 c^6}{4}+8 c^5-54 c^4-48 c^3+100 c^2+64 c\right) x^2\\&+\left(-\frac{13 c^6}{2}-10 c^5+50 c^4+32 c^3-192 c^2+32 c+384\right)v\\\leq&\frac{5 c^6}{4}-4 c^5-12 c^4+16 c^3+48 c^2\\&+\left(c^6-4 c^5-8 c^4+32 c^3+16 c^2-64 c\right) x^4\\&+\left(-7 c^6+10 c^5+40 c^4-32 c^3-32 c^2-32 c-64\right) x^3\\&+\left(\frac{29 c^6}{4}+8 c^5-54 c^4-48 c^3+100 c^2+64 c\right) x^2\\&+\left(-\frac{13 c^6}{2}-10 c^5+50 c^4+32 c^3-192 c^2+32 c+384\right)\\=&384+32 c-\frac{21 c^6}{4}-14 c^5+38 c^4+48 c^3-144 c^2\\&+\left(c^6-4 c^5-8 c^4+32 c^3+16 c^2-64 c\right) x^4\\&+\left(-7 c^6+10 c^5+40 c^4-32 c^3-32 c^2-32 c-64\right) x^3\\&+\left(\frac{29 c^6}{4}+8 c^5-54 c^4-48 c^3+100 c^2+64 c\right) x^2:=h(c,x). 
				\end{split}
			\end{equation*}
			
			For $c\in\left(\frac{4}{\sqrt{7}},\sqrt{\frac{5}{2}}\right]$ and $x\in(0,1)$
			\begin{equation*}
				h(c,x)\leq295+28x^2-81x^3-8x^4<296.
			\end{equation*}
			
			For $c\in\left(\sqrt{\frac{5}{2}},2\right)$ and $x\in(0,1)$
			\begin{equation*}
				h(c,x)\leq282+17x^2-0x^3+1x^4<300.
			\end{equation*}
			
			Hence,from \textbf{Case D1 } and \textbf{Case D2 }$\vartheta (c,x,y)<320.$
			\\
			In review of cases \textbf{A}, \textbf{B}, \textbf{C} and \textbf{D}, we obtain
			\begin{equation}
				\max\bigg\{\vartheta  (c,x,y)\leq320:c\in[0,2],x\in[0,1]~\text{and}~y\in[0,1]\bigg\}.
			\end{equation}
			Symplifying the expressions (2.10) and (2.14), we get
			\begin{equation}
				\Big|H_{3,1}(f^{-1})\Big|\leq\frac{1}{16}.
			\end{equation}
			For $f_0\in\mathcal{F}$, we obtain $a_2=a_4=0,~	a_3=1/2~\text{and}~a_5=3/8$, further we have, $t_2=t_4=0,~	t_3=-1/2~\text{and}~t_5=3/8$,  which follows the result.
		\end{proof}
		\textbf{Data Availability:}~My manuscript has no associate data

	\end{document}